\documentclass[a4paper,11pt]{article}
\usepackage[hidelinks]{hyperref}
\hypersetup{
    colorlinks=true,
    linktoc=all,
    linkcolor=red,     
   citecolor=blue,     
}
\usepackage{tikz}
\usepackage{subfigure}
\usepackage{graphicx}
\usepackage{amssymb}
\usepackage{latexsym, bm}
\usepackage{multicol}
\usepackage{indentfirst}
\usepackage{amsfonts}
\usepackage{amsmath}

\usepackage{diagbox}
\usepackage[utf8]{inputenc}
\usepackage{authblk}
\usepackage{setspace}
\fontfamily{cmr}\fontsize{9}{11}\selectfont

\newtheorem{theorem}{Theorem}[section]
\newtheorem{lemma}[theorem]{Lemma}
\newtheorem{claim}[theorem]{Claim}

\newtheorem{conjecture}[theorem]{Conjecture}

\newtheorem{observation}[theorem]{Observation}
\newcommand{\ignore}[1]{}

\begin{document}

\begin{spacing}{1.0}
\date{}
\title{An improved double-exponential lower bound for $r_4(5,n)$\thanks{After submitting the manuscript to arXiv.org, the first and third authors learned that the second and fourth authors, as well as Longma Du, Xinyu Hu, Ruilong Liu, and Guanghui Wang, had independently obtained the same result. Following discussions, the authors decided to write a joint paper. The current manuscript is mainly based on the original version by the first and third authors. }}
		
\author{Chunchao Fan,\footnote{Center for Discrete Mathematics, Fuzhou University, Fuzhou, 350108, P.~R.~China. Email: {\tt 1807951575@qq.com}.}
			\;\; Mingze Li,\footnote{School of Mathematical Sciences, University of Science and Technology of China, Hefei
230026, China. Email: {\tt lmz10@mail.ustc.edu.cn}} \;\; Qizhong Lin,\footnote{Center for Discrete Mathematics, Fuzhou University, Fuzhou, 350108, P.~R.~China. Email: {\tt linqizhong@fzu.edu.cn}. Supported by National Key R\&D Program of China (Grant No. 2023YFA1010202) and NSFC (No.\ 12571361).} \;\; Bo Ning\footnote{College of Computer Science, Nankai University, Tianjin, 300350, P.R. China. Email: \texttt{bo.ning@nankai.edu.cn}. Research supported in part by NSFC (No. 12371350) and the Fundamental Research Funds for the Central Universities, Nankai University (No. 63263259).}
             }
\maketitle

\begin{abstract}
The Ramsey number $r_k(s,n)$ is the smallest integer $N$ such that every $N$-vertex $k$-graph contains either a copy of $K_s^{(k)}$ or an independent set of size $n$. 
A well-known conjecture of Erd\H{o}s and Hajnal states that for any fixed $4\le k<s$, $r_k(s,n)\ge \operatorname{twr}_{k-1}(\Omega(n)).$
At present, only the last two cases of this conjecture remain open, namely $r_4(5,n)\ge2^{2^{\Omega(n)}}$ and $r_4(6,n)\ge2^{2^{\Omega(n)}}$.
Recently, Du, Hu, Liu, and Wang achieved a breakthrough by proving $r_4(5,n)\ge 2^{2^{\Omega(n^{1/7})}}$, which is the first double-exponential lower bound for $r_4(5,n)$. In this note, we improve this to $2^{2^{\Omega(n^{1/5})}}$ by modifying their  construction and reducing the greedy selection of local maxima from seven layers to five, thereby making further progress towards the Erd\H{o}s-Hajnal conjecture.

\end{abstract}
	
\section{Introduction}
We write $K^{(k)}_n$ for the complete $k$-uniform hypergraph ($k$-graph for short) on $n$ vertices.
The Ramsey number $r_k(s,n)$ is the smallest integer $N$ such that every $N$-vertex $k$-graph contains either a copy of $K_s^{(k)}$ or an independent set of size $n$.  Estimating the Ramsey number $r_k(s,n)$ is a fundamental and notoriously difficult problem in combinatorics, which has attracted great attention since 1935.

For the diagonal cases, it is well known that $2^{n/2}<r_2(n,n)<2^{2n}$ due to Erd\H{o}s and Szekeres \cite{E-S} and Erd\H{o}s \cite{E-1}.
Following several important improvements by Thomason \cite{Thom}, Conlon \cite{C-1}, and Sah \cite{sah}, the first exponential improvement was achieved by Campos, Griffiths, Morris, and Sahasrabudhe \cite{C-G-M-S}, who showed that for some constant $\varepsilon > 0$ and all sufficiently large $n$, $r_2(n,n)<(4-\varepsilon)^n$. More recently, Gupta, Ndiaye, Norin, and Wei \cite{G-N-N-L} improved the upper bound to $3.8^{n+o(n)}$. 

However, for $k\ge3$, there still exists an exponential gap between the best known
upper and lower bounds for $r_k(n,n)$. In particular, Erd\H{o}s, Hajnal, and Rado \cite{E-H-R,E-R-2} showed that
$$\textrm{twr}_{k-1}(\Omega(n^2))<r_k(n,n)<\textrm{twr}_k(O(n)),$$
where the tower function is defined recursively as $\textrm{twr}_1(x)=x$ and $\textrm{twr}_{i+1}(x)=2^{\textrm{twr}_i(x)}$. 
Erd\H{o}s, Hajnal, and Rado \cite{E-H-R} conjectured that $r_k(n,n)=\textrm{twr}_k(\Theta(n))$.

Off-diagonal Ramsey numbers $r_k(s,n)$ have also been extensively studied where $k$ and $s$ are fixed.
The classic results of Kim \cite{kim} and Ajtai, Koml\'{o}s, and Szemer\'{e}di \cite{A-K-S} determined $r_2(3,n)=\Theta(n^2/\log n)$. The best upper bound $r_2(3,n)\le(1+o(1))n^2/\log n$ is due to Shearer \cite{she}.
Following several important advances by Bohman \cite{B-1}, Fiz Pontiveros, Griffiths, and Morris \cite{FP-G-M} and Campos, Jenssen, Michelen, and Sahasrabudhe \cite{C-J-M-S}, Hefty, Horn, King, and Pfender \cite{H-H-K-P} obtained $r_2(3,n)\ge(\frac{1}{2}+o(1))n^2/\log n$.
Recently, a breakthrough result of Mattheus and Verstra\"{e}te \cite{M-V} showed that $r_2(4,n)\ge\Omega(\frac{n^3}{\log^4n})$.
Generally, the previous best bounds \cite{A-K-S,B-K,lrz,sp75} for $r_2(s,n)$ with fixed $s\ge 5$ are $\tilde{\Omega}(n^{\frac{s+1}{2}})\le r_2(s,n)\le(1+o(1))(n^{s-1}/\log^{s-2} n)$, where $\tilde{\Omega}$ suppresses factors of the form $(\log n)^{\Theta(1)}$.
In particular, Erd\H{o}s \cite{C-G-E} conjectured that $r_2(s,n)> n^{s-1}/\log ^cn$ for some suitable constant $c > 0$. Remarkably, Ma, Shen, and Xie \cite{M-S-X} obtained the first exponential improvement to Erd\H{o}s' lower bound for $r_2(n,Cn)$ with constant $C>1$.
Hunter, Milojevi\'c, and Sudakov \cite{H-M-S} (independently Sahasrabudhe \cite{saha}) gave a simple proof based on Gaussian random graphs.

For $3$-graphs, improving the upper bound by Erd\H{o}s and Rado \cite{E-R-2} and the lower bound by Erd\H{o}s and Hajnal \cite{E-H-Con},  Conlon, Fox, and Sudakov \cite{C-F-S} obtained that for $s\ge4$, $$2^{\Omega (n\log n)}\le r_3(s,n)\le2^{O(n^{s-2}\log n)}.$$
For $k\ge4$, Erd\H{o}s and Rado \cite{E-R-2} showed that $$r_k(s,n)\le \textrm{twr}_{k-1}(n^{O(1)}).$$
By applying the Erd\H{o}s-Hajnal stepping-up lemma, it follows that $r_k(s,n)\ge \textrm{twr}_{k-1}(\Omega(n))$
for $k\ge4$ and $s\ge2^{k-1}-k+3$, in particular, $r_4(7,n)\ge 2^{2^{\Omega(n)}}$.

Erd\H{o}s and Hajnal made the following conjecture.
\begin{conjecture}[Erd\H{o}s and Hajnal \cite{E-H-Con}]\label{E-H-C}
    For fixed $4\le k<s$,
$r_k(s,n)\ge \operatorname{twr}_{k-1}(\Omega(n)).$
\end{conjecture}

 Conlon, Fox, and Sudakov \cite{cfs} verified this conjecture for  $s\ge\lceil5k/2\rceil-3$. Subsequently, Mubayi and Suk \cite{M-S} and independently Conlon, Fox, and Sudakov (unpublished) verified this conjecture for all $s\ge k+3$. 
Improving the bounds of $r_4(5,n)$ and $r_4(6,n)$ in \cite{M-S,M-S-2},  Mubayi and Suk \cite{M-S-1} proved that for $s\ge k+2$, $r_4(6,n)\ge2^{2^{\Omega(n^{1/5})}}$ and consequently $r_k(k+2,n)\ge \textrm{twr}_{k-1}(\Omega(n^{1/5}))$. This bound was later improved  by Du, Hu, Liu, and Wang \cite{D-H-L-W-0} to $r_k(k+2,n)\ge \textrm{twr}_{k-1}(\Omega(n^{1/2}))$. For the remaining case $s=k+1$, Mubayi and Suk \cite{M-S-1} showed $r_4(5,n)\ge2^{n^{\Omega(\log n)}},$ and consequently $r_k(k+1,n)\ge \textrm{twr}_{k-2}(n^{\Omega(\log n)})$. To close the exponential gap between the best known lower and upper bounds for $r_k(k+1,n)$ for $k\ge4$, it suffices to show that there exists some constant $c>0$ such that $$r_4(5,n)\ge2^{2^{\Omega(n^c)}}$$ from the stepping-up lemma \cite[Lemma 2.5]{M-S-2} (one can see \cite[Lemma 2.1]{M-S-1} for general cases). 

\medskip
In particular, Mubayi and Suk posed the following slightly weak conjecture.
\begin{conjecture}[Mubayi and Suk \cite{M-S-1}]\label{M-S-conj}
    For $n\ge5$, there is an absolute constant $c > 0$ such that
$r_4(5,n)\ge 2^{2^{n^{c}}}.$
\end{conjecture}

Du, Hu, Liu, and Wang recently confirmed the above conjecture.

\begin{theorem}[Du, Hu, Liu, and Wang \cite{D-H-L-W}]\label{center-2}
     For all $n\ge 5$, we have $r_4(5,n)\ge 2^{2^{\Omega(n^{1/7})}}.$
\end{theorem}

In this short note, we further improve this lower bound by modifying their construction.
\begin{theorem}\label{main}
For all $n\geq 5$, we have
$
r_4(5, n) \ge  2^{2^{\Omega(n^{1/5})}}.
$
\end{theorem}

It is important to note that whether $r_4(5,n) \ge 2^{2^{\Omega(n)}}$ and $r_4(6,n) \ge 2^{2^{\Omega(n)}}$ hold remains an open problem, which is exactly what the Erd\H{o}s-Hajnal conjecture predicts for these cases. The current best results only achieve $2^{2^{\Omega(n^{1/5})}}$ from our result and $2^{2^{\Omega(n^{1/2})}}$ from \cite{D-H-L-W-0}, which are still a considerable distance from the conjectured $2^{2^{\Omega(n)}}$.

\section{Notations and basic properties}\label{B-p}


We will use the following notations and definitions unless otherwise stated.

Let $V = \{0, 1, \ldots, 2^N - 1\}$. For any $a \in V$, we write $a = \sum_{i=0}^{N-1} a(i) 2^i$ where $a(i) \in \{0,1\}$ for each $i$. For $a \neq b$, let $\delta(a,b)$ denote the largest $i$ such that $a(i) \neq b(i)$. We always denote a subset by $(a_1, a_2, \ldots, a_r)$ with $a_1 < a_2 < \cdots < a_r$.

Given any set $A = (a_1, a_2, \ldots, a_r)$ of vertices, we write for $1 \le i \le r-1$
\[
\delta_i = \delta(a_i, a_{i+1}), \qquad \text{and} \qquad \delta(A) = \{\delta_1, \delta_2, \dots, \delta_{r-1}\}.
\]

We have the following properties \cite{grs}.

\begin{description}\label{p-ab}
\item[Property A:] For every triple $a < b < c$, $\delta(a,b) \neq \delta(b,c)$.
\item[Property B:] For $a_1 < \cdots < a_r$, $\delta(a_1, a_r) = \max_{1 \le i \le r-1} \delta_i$.
\end{description}

We will also use the following property, see e.g. \cite{F-H-L-L,H-L-L-W-1,M-S-1}.

\begin{description}\label{p-d}
\item[Property C:] For every $4$-tuple $a_1 < \cdots < a_4$, if $\delta_1 > \delta_2$, then $\delta_1 \neq \delta_3$.
\end{description}

\noindent
{\bf Proof.} Suppose, to the contrary, that $\delta_1 = \delta_3$. Then, by Property B, $\delta(a_1, a_3) = \delta_1 = \delta_3 = \delta(a_3, a_4)$. This contradicts Property A for the triple $a_1 < a_3 < a_4$. \hfill $\Box$

\medskip
For any set $A = (a_1,\ldots, a_r)$, we call $\delta_i$ a \textit{local minimum} of $A$ if $\delta_{i-1} > \delta_i < \delta_{i+1}$, a \textit{local maximum} if $\delta_{i-1} < \delta_i > \delta_{i+1}$, and a \textit{local extremum} if it is either a local minimum or a local maximum. We say that the sequence $\{\delta_1,\delta_2,\dots,\delta_{r-1}\}$ is \textit{monotone} if $\delta_1<\delta_2<\cdots<\delta_{r-1}$ (monotone increasing) or $\delta_1>\delta_2>\cdots>\delta_{r-1}$ (monotone decreasing).

\section{New lower bound for $r_4(5,n)$}\label{pf}

We need the following result by Du, Hu, Liu, and Wang \cite{D-H-L-W}.
\begin{lemma}[Du, Hu, Liu, and Wang \cite{D-H-L-W}]\label{phi}
For every $n\ge 5$, there exists an absolute constant $c_0>0$ such that the following holds. There is a red/blue coloring $\phi$ of the pairs of $\{0,1,\ldots,\lfloor 2^{c_0n}\rfloor-1\}$ with the property that every $n$-set $A\subset \{0,1,\ldots,\lfloor 2^{c_0n}\rfloor-1\}$ contains a $3$-tuple $a_i<a_j<a_k$ satisfying $\phi(a_i,a_j)=\phi(a_j,a_k)\ne\phi(a_i,a_k).$
\end{lemma}
{\bf Proof sketch.} We randomly color each pair of elements independently red/blue. Fix a partial Steiner $(n,3,2)$-system \cite{E-H-2} with $\Omega(n^2)$ edges such that no two edges share a pair. Since the edges are pair-disjoint, the corresponding bad events (a $3$-tuple $a_i<a_j<a_k$ that violates $\phi(a_i,a_j)=\phi(a_j,a_k)\ne\phi(a_i,a_k)$) are independent, giving an upper bound of $(3/4)^{\Omega(n^2)}$ on the probability that a given $n$-set contains no good triple. There are $\binom{D}{n}$ such $n$-sets where $D \approx 2^{c_0 n}$. Choosing $c_0$ sufficiently small makes the expected number of $n$-sets without a good triple less than $1$. Hence, there exists a coloring in which every $n$-set contains such a ``good'' triple. \hfill$\Box$

\medskip

Let $c_0>0$ be the constant from Lemma \ref{phi}, and let $U=\{0,1,\ldots,\lfloor2^{c_0n}\rfloor-1\}$ and $\phi:\binom{U}{2}\to \{\text{red},\text{blue}\}$ be a $2$-coloring of the pairs of $U$ satisfying the properties given in the lemma. Now, let $N=2^{\lfloor2^{c_0n}\rfloor}$ and $V(H)=\{0,1,\ldots,N-1\}$. Then we shall use the coloring
$\phi$ to produce a $K^{(4)}_5$-free $4$-graph $H$ on $V(H)$ with $\alpha(H)<2^6n^5+1$ as follows. For any $4$-tuple $e=(v_1,v_2,v_3,v_4)$ of $V(H)$, set $e\in E(H)$ if and only if one of the following holds:
\begin{enumerate}
\item[\textbf{(i)}] $\delta_1,\delta_2,\delta_3$ form a monotone sequence and $\phi(\delta_1,\delta_2)=\phi(\delta_2,\delta_3)\ne\phi(\delta_1,\delta_3)$.

\item[\textbf{(ii)}] $\delta_1>\delta_2<\delta_3$, $\delta_1>\delta_3$ and $\phi(\delta_1,\delta_2)\ne\phi(\delta_2,\delta_3)$.
\item[\textbf{(iii)}] $\delta_1>\delta_2<\delta_3$, $\delta_1<\delta_3$ and $\phi(\delta_1,\delta_2)=\phi(\delta_1,\delta_3)=\phi(\delta_2,\delta_3)$.

\end{enumerate}

\subsection{$H$ is $K_5^{(4)}$-free }

In this subsection, we show that $H$ is $K_5^{(4)}$-free.
To see this, suppose to the contrary that $P =(v_1, \ldots, v_5)$ induces a $K_5^{(4)}$ in $H$. This will lead to a contradiction. Recall that $\delta(P)=\{\delta_1,\delta_2,\delta_3,\delta_4\}$.

\begin{claim}\label{no-mono}
$\delta(P)$ is not a monotone sequence.
\end{claim}
\noindent\textbf{Proof.}
Suppose, to the contrary, that $\delta(P)$ is increasing; the decreasing case is analogous. By \textbf{(i)} and $\delta(v_1,v_2,v_3,v_4)=(\delta_1<\delta_2<\delta_3)$, we have $\phi(\delta_2,\delta_3)\ne\phi(\delta_1,\delta_3)$. Moreover, by \textbf{(i)} and $\delta(v_2,v_3,v_4,v_5)=(\delta_2<\delta_3<\delta_4)$, we have $\phi(\delta_2,\delta_3)=\phi(\delta_3,\delta_4)$. Therefore, $(v_1,v_2,v_4,v_5)\notin E(H[P])$ since $\delta(v_1,v_2,v_4,v_5)=(\delta_1<\delta_3<\delta_4)$ and $\phi(\delta_1,\delta_3)\ne\phi(\delta_3,\delta_4)$, a contradiction.
\hfill$\Box$

\medskip
By Claim \ref{no-mono}, there exists $i\in \{2,3\}$ such that $\delta_i$ is a local extremum. If $\delta_i$ is a local maximum, then $(v_{i-1},v_i,v_{i+1},v_{i+2})\notin E(H[P])$ since $\delta(v_{i-1},v_i,v_{i+1},v_{i+2})=(\delta_{i-1}<\delta_i>\delta_{i+1})$, a contradiction. Thus, we may assume that $\delta_i$ is a local minimum.

\medskip
\textbf{Case 1.~} $i=2$, i.e., $\delta(P)=(\delta_1>\delta_2<\delta_3<\delta_4)$.
\medskip

Since $(v_2,v_3,v_4,v_5)\in E(H[P])$ and $\delta(v_2,v_3,v_4,v_5)=(\delta_2<\delta_3<\delta_4)$, \textbf{(i)} gives $\phi(\delta_2,\delta_3)=\phi(\delta_3,\delta_4)\ne\phi(\delta_2,\delta_4)$.
If $\delta_1<\delta_3$, then applying \textbf{(iii)} to $(v_1,v_2,v_3,v_5)$ and $(v_1,v_2,v_3,v_4)$ gives $\phi(\delta_1,\delta_2)=\phi(\delta_2,\delta_4)$ and $\phi(\delta_1,\delta_2)=\phi(\delta_2,\delta_3)$. This contradicts $\phi(\delta_2,\delta_3)\ne\phi(\delta_2,\delta_4)$.
If $\delta_1>\delta_4$, then applying \textbf{(ii)} to the same two $4$-tuples yields
$\phi(\delta_1,\delta_2)\ne\phi(\delta_2,\delta_4)$ and $\phi(\delta_1,\delta_2)\ne\phi(\delta_2,\delta_3)$.
Since only two colors are available, we must have $\phi(\delta_2,\delta_3)=\phi(\delta_2,\delta_4)$, again a contradiction.

Therefore, by Property C, we have $\delta_4>\delta_1>\delta_3$.
Then applying \textbf{(iii)} to $(v_1,v_2,v_3,v_5)$ and $(v_1,v_2,v_4,v_5)$ gives $\phi(\delta_1,\delta_4)=\phi(\delta_2,\delta_4)$ and $\phi(\delta_1,\delta_4)=\phi(\delta_3,\delta_4)$. This contradicts $\phi(\delta_3,\delta_4)\ne\phi(\delta_2,\delta_4)$.

\medskip
\textbf{Case 2.~} $i=3$, i.e., $\delta(P)=(\delta_1>\delta_2>\delta_3<\delta_4)$.
\medskip

Since $(v_1,v_2,v_3,v_4)\in E(H[P])$ and $\delta(v_1,v_2,v_3,v_4)=(\delta_1>\delta_2>\delta_3)$, \textbf{(i)} gives $\phi(\delta_1,\delta_2)=\phi(\delta_2,\delta_3)\ne\phi(\delta_1,\delta_3)$.
If $\delta_2>\delta_4$,  applying \textbf{(ii)} to $(v_1,v_3,v_4,v_5)$ gives $\phi(\delta_3,\delta_4)\ne\phi(\delta_1,\delta_3)$, since $\delta(v_1,v_3,v_4,v_5)=(\delta_1>\delta_3<\delta_4)$ and $\delta_1>\delta_4$. Similarly, applying \textbf{(ii)} to $(v_2,v_3,v_4,v_5)$ gives $\phi(\delta_3,\delta_4)\ne\phi(\delta_2,\delta_3)$, contradicting $\phi(\delta_2,\delta_3)\neq\phi(\delta_1,\delta_3)$.
If $\delta_1<\delta_4$,  applying \textbf{(iii)} to $(v_1,v_3,v_4,v_5)$ gives $\phi(\delta_3,\delta_4)=\phi(\delta_1,\delta_3)$. Similarly, applying \textbf{(iii)} to $(v_2,v_3,v_4,v_5)$ gives $\phi(\delta_3,\delta_4)=\phi(\delta_2,\delta_3)$, contradicting $\phi(\delta_2,\delta_3)\neq\phi(\delta_1,\delta_3)$.

Therefore, by Property C, we have $\delta_1>\delta_4>\delta_2$.
Now $\delta(v_1,v_2,v_3,v_5)=(\delta_1>\delta_2<\delta_4)$ and $\delta_1>\delta_4$, so \textbf{(ii)} gives $\phi(\delta_1,\delta_2)\neq\phi(\delta_2,\delta_4)$. On the other hand, applying \textbf{(iii)} to $(v_2,v_3,v_4,v_5)$ gives $\phi(\delta_2,\delta_3)=\phi(\delta_2,\delta_4)$, contradicting $\phi(\delta_1,\delta_2)=\phi(\delta_2,\delta_3)$.

\medskip

This shows that $H$ is $K_5^{(4)}$-free.

\subsection{$\alpha(H)<2^{6}n^{5}+1$}
Now we show that $\alpha(H)<2^{6}n^{5}+1$. Suppose to the contrary that there is a set $Q=(v_1,v_2,\ldots ,v_m)$ of $m=2^{6}n^{5}+1$ vertices that induces an independent set in $H$. Recall that $\delta_i=\delta(v_i,v_{i+1})$, and hence $\delta(Q)=\{\delta_1,\delta_2,\dots,\delta_{m-1}\}$.

\begin{lemma}\label{no-mono-n}
There is no monotone subsequence $\{\delta_{i_1},\delta_{i_2},\dots,\delta_{i_n}\}\subset\delta(Q)$ such that for any $a,b,c \in [n]$ with $a<b<c$, there exists $(u_1, u_2, u_3, u_4)\subset Q$ such that $\delta(u_1, u_2, u_3, u_4)=(\delta_{i_a},\delta_{i_b},\delta_{i_c})$.
\end{lemma}

\noindent\textbf{Proof.}
Without loss of generality, suppose to the contrary that $\{\delta_{i_1},\delta_{i_2},\dots,\delta_{i_n}\}$ is such a monotone increasing subsequence. By Lemma \ref{phi}, there exists a $3$-tuple $\delta_{i_a},\delta_{i_b},\delta_{i_c}$ with $\delta_{i_a}<\delta_{i_b}<\delta_{i_c}$ such that $$\phi(\delta_{i_a},\delta_{i_b})=\phi(\delta_{i_b},\delta_{i_c})\neq\phi(\delta_{i_a},\delta_{i_c}).$$
By the assumption on the subsequence, there exists $(u_1, u_2, u_3, u_4)\subset Q$ with $\delta(u_1, u_2, u_3, u_4)=(\delta_{i_a},\delta_{i_b},\delta_{i_c})$. Then  by \textbf{(i)}, $(u_1, u_2, u_3, u_4)$ forms an edge in $H$, contradicting that $Q$ is independent.
\hfill$\Box$

\medskip
Let $\beta_i=\frac{m-1}{(2n)^i}$, for $i\in[0,5]$. Since $m-1=2^6n^5$, each $\beta_i$ is an integer.
For $t\in[5]$, we will greedily construct \emph{$t$-layer local maxima sequences} $\Delta^{(t)}$ such that $\Delta^{(t)}\subset\Delta^{(t-1)}$, starting with $\Delta^{(0)}=\delta(Q)$, and such that the following property holds. (We do not require that the elements of $\Delta^{(t)}$ be distinct.)

\begin{itemize}
\item[($\ast$)] For two consecutive elements $\delta_a$, $\delta_b\in\Delta^{(t)}$, we have $\delta_x<\max\{\delta_a ,\delta_b\}$ for all $a<x<b$, and hence $\delta_a\neq\delta_b$.
\end{itemize}

For $t\ge1$, assume that $\Delta^{(t-1)}$ has been constructed with property ($\ast$). We now define $\Delta^{(t)}$ as the first $\beta_t$ local maxima with respect to  $\Delta^{(t-1)}$. In what follows, we abbreviate ``with respect to" as ``w.r.t." for convenience.

We first claim that $\Delta^{(t-1)}$ contains no monotone consecutive subsequence of length $n$. On the contrary, such a subsequence $Q'$ exists. Without loss of generality, assume $Q'$ is increasing. For any $\delta_{j_1},\delta_{j_2},\delta_{j_3}\in Q'$ with $j_1<j_2<j_3$, by the first part of property ($\ast$) of $\Delta^{(t-1)}$ and Property B, we have $\delta(v_{j_1},v_{j_1+1},v_{j_2+1},v_{j_3+1})=(\delta_{j_1}<\delta_{j_2}<\delta_{j_3})$, which contradicts Lemma \ref{no-mono-n}. Therefore, by the second part of property ($\ast$), we may take $\Delta^{(t)}$ to be the first $\frac{\beta_{t-1}}{2n} = \beta_t$ local maxima (w.r.t. $\Delta^{(t-1)}$). Consequently, $\Delta^{(t)} \subset \Delta^{(t-1)}$.

To show the  property ($\ast$) for $\Delta^{(t)}$,  we consider two consecutive elements $\delta_a$, $\delta_b\in\Delta^{(t)}$ and  we may assume that $\delta_{a}=\delta_{i_0},\delta_{i_1},\delta_{i_2},\cdots,\delta_{i_j},\delta_{b}=\delta_{i_{j+1}}$ are consecutive elements in $\Delta^{(t-1)}$. Note that $\delta_{a}$ and $\delta_{b}$ are consecutive local maxima (w.r.t. $\Delta^{(t-1)}$), we have $\delta_{i_\ell}<\max\{\delta_a, \delta_b\}$ for $\ell\in[j]$. Furthermore, it follows from the inductive hypothesis that $\delta_x<\max\{\delta_{i_\ell},\delta_{i_{\ell+1}}\}$ for all $i_\ell<x<i_{\ell+1}$ and $0\leq\ell \leq j$, and so $\delta_x<\max\{\delta_{i_\ell},\delta_{i_{\ell+1}}\}\leq\max\{\delta_{a},\delta_{b}\}$. Thus, $\delta_x<\max\{\delta_{a},\delta_{b}\}$ for all $a<x<b$. Moreover, Property A implies that $\delta_a\neq\delta_b$, as desired. Otherwise $\delta(v_a,v_b)=\delta_a=\delta_b=\delta(v_b,v_{b+1})$, a contradiction.

\medskip
Now let $t\in[5]$ and $\delta_j\in\Delta^{(t)}\setminus\Delta^{(t+1)}$ where $\Delta^{(6)}=\emptyset$.
Then $\delta_j$ is a local maximum (w.r.t. $\Delta^{(t-1)}$). Denote by $\delta_{j^-}$ and $\delta_{j^+}$ the \textbf{closest} elements to the left and right of $\delta_j$ in the sequence $\Delta^{(t-1)}$, respectively. Thus,
$\delta_{j^-},\delta_{j^+}<\delta_{j}$, and in particular $\delta_{j^-},\delta_{j^+}\in \Delta^{(t-1)}\setminus\Delta^{(t)}$. By repeatedly applying property ($\ast$) in the above greedy construction, we obtain the following observation.

\begin{observation}\label{observation-1}
For $t\in [5]$ and  $\delta_j\in \Delta^{(t)}\setminus\Delta^{(t+1)}$,
we have $\delta_x<\delta_{j}$ for each $x\in[j^-,j^+]\setminus\{j\}$.
\end{observation}

Note that $|\Delta^{(5)}|=\beta_5=2$. Let $\Delta^{(5)}=\{\delta_{a_1},\delta_{a_2}\}$ where $a_1<a_2$.
By property ($\ast$), it follows that $\delta_{a_1}\neq\delta_{a_2}$.

We first consider the case that $\delta_{a_1}<\delta_{a_2}$.
Define $\delta_{b_1}=\delta_{a_1^+}\in\Delta^{(4)}$,  $\delta_{b_2}=\delta_{b_1^+}\in\Delta^{(3)}$.
Then $a_1<b_1<b_2<a_2$ and $\delta_{b_2}<\delta_{b_1}<\delta_{a_1}<\delta_{a_2}$. By the pigeonhole principle, we have either $\phi(\delta_{a_1},\delta_{a_2})=\phi(\delta_{b_1},\delta_{a_2})$, or $\phi(\delta_{a_1},\delta_{a_2})=\phi(\delta_{b_2},\delta_{a_2})$, or $\phi(\delta_{b_1},\delta_{a_2})=\phi(\delta_{b_2},\delta_{a_2})$.
We give the details for the case
$\phi(\delta_{a_1},\delta_{a_2})=\phi(\delta_{b_2},\delta_{a_2})$.
The other two cases are handled in exactly the same way: for example, if $\phi(\delta_{a_1},\delta_{a_2})=\phi(\delta_{b_1},\delta_{a_2})$,
then one replaces $b_2$ with $b_1$, and starts the subsequent construction from $b_1$. Now we let $\delta_c=\delta_{b_2^-}\in \Delta^{(2)}$, $\delta_d=\delta_{c^+}\in \Delta^{(1)}$, $\delta_e=\delta_{d^-}\in \delta(Q)$. Recall that $Q=(v_1,v_2,\ldots ,v_m)$ induces an independent set in $H$.

\begin{claim}\label{ind-1}
(i) For every $c\le x<b_2$, we have $\phi(\delta_x,\delta_{b_2})\ne\phi(\delta_{b_2},\delta_{a_2})$.

\medskip
(ii) For every $e\le x<d$, we have $\phi(\delta_x,\delta_{d})\ne\phi(\delta_{d},\delta_{b_2})$.
\end{claim}

\noindent\textbf{Proof.}
(i) Suppose, to the contrary, that $\phi(\delta_x,\delta_{b_2})=\phi(\delta_{b_2},\delta_{a_2})$ for some $c\le x<b_2$. We first show that $\phi(\delta_x,\delta_{a_2})=\phi(\delta_{b_2},\delta_{a_2})$. Indeed, otherwise $\phi(\delta_x,\delta_{a_2})\neq\phi(\delta_{b_2},\delta_{a_2})=\phi(\delta_x,\delta_{b_2})$ which together with Observation \ref{observation-1} and \textbf{(i)} would imply $(v_x,v_{x+1},v_{b_2+1},v_{a_2+1})\in E(H[Q])$, a contradiction.
Next, $\phi(\delta_{a_1},\delta_{x})=\phi(\delta_{x},\delta_{b_2})$;
otherwise $\phi(\delta_{a_1},\delta_{x})\neq\phi(\delta_{x},\delta_{b_2})$, and Observation \ref{observation-1} together with \textbf{(ii)} would give $(v_{a_1},v_{x},v_{x+1},v_{b_2+1})\in E(H[Q])$, again a contradiction.
Recall that $\phi(\delta_{a_1},\delta_{a_2})=\phi(\delta_{b_2},\delta_{a_2})$. It follows that $\phi(\delta_{a_1},\delta_{x})=\phi(\delta_x,\delta_{a_2})=\phi(\delta_{a_1},\delta_{a_2})$.
Then Observation \ref{observation-1} and \textbf{(iii)} implies $(v_{a_1},v_x,v_{x+1},v_{a_2+1})\in E(H[Q])$, a contradiction.

\medskip
(ii) By (i) and the fact that $\phi$ uses only two colors, it follows that $\phi(\delta_c,\delta_{b_2})=\phi(\delta_d,\delta_{b_2})$.
By a similar argument as above, we can obtain the desired result.
\hfill$\Box$

\medskip
By Claim~\ref{ind-1}(i) and the fact that $\phi$ uses only two colors, it follows that $\phi(\delta_e,\delta_{b_2})=\phi(\delta_d,\delta_{b_2})$. Then Claim~\ref{ind-1}(ii) forces $\phi(\delta_e,\delta_{d})\neq\phi(\delta_d,\delta_{b_2})$.
Then applying \textbf{(ii)} to $(v_{a_1},v_e,v_{e+1},v_{d+1})$ gives $\phi(\delta_{a_1},\delta_e)=\phi(\delta_e,\delta_{d})$.
Therefore $\phi(\delta_{a_1},\delta_e)\neq\phi(\delta_e,\delta_{b_2})$.
Now observe that $\delta(v_{a_1},v_e,v_{e+1},v_{b_2+1})=(\delta_{a_1}>\delta_e<\delta_{b_2})$ and $\delta_{a_1}>\delta_{b_2}.$
Thus by \textbf{(ii)}, $(v_{a_1},v_e,v_{e+1},v_{b_2+1})\in E(H[Q])$, a contradiction.

\medskip
Now we may assume $\delta_{a_1}>\delta_{a_2}$.
We set $\delta_{b}=\delta_{a_2^-}\in\Delta^{(4)}$,  $\delta_{c}=\delta_{b^+}\in\Delta^{(3)}$, $\delta_d=\delta_{c^-}\in \Delta^{(2)}$, $\delta_e=\delta_{d^+}\in \Delta^{(1)}$ and $\delta_f=\delta_{e^-}\in \delta(Q)$.

\begin{claim}\label{ind-2}
(i) For every $b< x\leq c$, we have $\phi(\delta_b,\delta_{x})\ne\phi(\delta_{a_1},\delta_b)$.

\medskip
(ii) For every $d< x\leq e$, we have $\phi(\delta_d,\delta_{x})\ne\phi(\delta_{b},\delta_d)$.
\end{claim}

\noindent\textbf{Proof.}
(i) Suppose, to the contrary, that $\phi(\delta_b,\delta_{x})=\phi(\delta_{a_1},\delta_b)$ for some $b< x\leq c$. We first show that $\phi(\delta_{a_1},\delta_x)=\phi(\delta_{a_1},\delta_b)$. Indeed, otherwise
$\phi(\delta_{a_1},\delta_x)\neq\phi(\delta_{a_1},\delta_b)=\phi(\delta_b,\delta_{x})$, which together with Observation \ref{observation-1} and \textbf{(i)} would imply $(v_{a_1},v_{b},v_{x},v_{x+1})\in E(H[Q])$, a contradiction.
Next, $\phi(\delta_{a_1},\delta_{b})=\phi(\delta_{b},\delta_{a_2})$;
otherwise $\phi(\delta_{a_1},\delta_{b})\neq\phi(\delta_{b},\delta_{a_2})$, and Observation \ref{observation-1} together with \textbf{(ii)} would give $(v_{a_1},v_{b},v_{b+1},v_{a_2+1})\in E(H[Q])$, again a contradiction.
Applying \textbf{(ii)} to $(v_{a_1},v_x,v_{x+1},v_{a_2+1})$ yields $\phi(\delta_{a_1},\delta_x)=\phi(\delta_x,\delta_{a_2})$.
Therefore, $\phi(\delta_{b},\delta_x)=\phi(\delta_{x},\delta_{a_2})=\phi(\delta_{b},\delta_{a_2})$.
Then \textbf{(iii)} implies $(v_{b},v_x,v_{x+1},v_{a_2+1})\in E(H[Q])$, a contradiction.

\medskip
(ii) By a similar argument as above, we can obtain the desired result.
\hfill$\Box$

\medskip
By Claim~\ref{ind-2}(i) and the fact that $\phi$ uses only two colors, it follows that $\phi(\delta_b,\delta_{d})=\phi(\delta_b,\delta_{f})$. Then Claim~\ref{ind-2}(ii) forces $\phi(\delta_d,\delta_{f})\neq\phi(\delta_b,\delta_{d})$.
Then applying \textbf{(ii)} to $(v_{d},v_f,v_{f+1},v_{e+1})$ gives $\phi(\delta_{d},\delta_f)=\phi(\delta_f,\delta_{e})$.
Therefore $\phi(\delta_{b},\delta_f)\neq\phi(\delta_f,\delta_{e})$.
Now observe that $\delta(v_{b},v_f,v_{f+1},v_{e+1})=(\delta_{b}>\delta_f<\delta_{e})$ and $\delta_{b}>\delta_{e}.$
Thus by \textbf{(ii)}, $(v_{b},v_f,v_{f+1},v_{e+1})\in E(H[Q])$, a contradiction.

This proves $\alpha(H)<2^6n^5+1$.

\medskip

Combining the two subsections, we have constructed a $4$-graph $H$ on
$N=2^{\lfloor 2^{c_0n}\rfloor}$ vertices that is $K_5^{(4)}$-free and satisfies
$\alpha(H)<2^6n^5+1$. Therefore $r_4(5,2^6n^5+1)>2^{\lfloor 2^{c_0n}\rfloor}$. Recall that $c_0>0$ is an absolute constant from Lemma \ref{phi}. Consequently, there exists an absolute constant $c>0$ such that for all $n\ge 5$,
$r_4(5,n)\ge 2^{2^{cn^{1/5}}}$. This completes the proof of Theorem~\ref{main}.\hfill$\Box$


\end{spacing}
\end{document}